\documentclass[12pt, a4paper]{article}
\usepackage[utf8]{inputenc}

\usepackage[margin = 1in]{geometry}

\usepackage{amsmath, amssymb, mathtools, url}
\usepackage[hidelinks]{hyperref}
\makeatletter
\let\@fnsymbol\@arabic  
\makeatother

\usepackage{xcolor}

\usepackage{amsthm}
\usepackage{cleveref}
\newtheorem{thm}{Theorem}[section]
\newtheorem{lm}[thm]{Lemma}
\newtheorem{res}[thm]{Result}
\newtheorem{crl}[thm]{Corollary}
\newtheorem{prop}[thm]{Proposition}

\theoremstyle{definition}

\newtheorem{rmk}[thm]{Remark}
\newtheorem{df}[thm]{Definition}
\newtheorem{nota}[thm]{Notation}
\newtheorem{ex}[thm]{Example}

\newcommand{\FF}{\mathbb F}

\newcommand{\QQ}{\mathbb Q}
\newcommand{\ZZ}{\mathbb Z}
\newcommand{\NN}{\mathbb N}

\newcommand{\vspan}[1]{\left \langle #1 \right \rangle}

\newcommand{\sett}[2]{ \left\{ #1 \, \, || \, \, #2 \right \} }

\newcommand{\one}{\mathbf 1}
\newcommand{\zero}{\mathbf 0}
\newcommand{\floor}[1]{\left \lfloor #1 \right \rfloor}
\newcommand{\ceil}[1]{\left \lceil #1 \right \rceil}

\newcommand{\mc}{\mathcal C}

\renewcommand{\mp}{\mathcal P}

\newcommand{\fqi}{\FF_q \cup \{\infty\}}

 \DeclareMathOperator{\supp}{supp}
 \DeclareMathOperator{\pr}{pr}
 \DeclareMathOperator{\wt}{wt}
 \DeclareMathOperator{\rk}{rk}

 \DeclareMathOperator{\PG}{PG}
    \newcommand{\pg}{\PG}
 \DeclareMathOperator{\AG}{AG}
    \newcommand{\ag}{\AG}
 \DeclareMathOperator{\PGL}{PGL}

 \DeclareMathOperator{\Sym}{Sym}
 \DeclareMathOperator{\sgn}{sgn}

\title{Multisets with few special directions and small weight codewords in Desarguesian planes}

\author{
 Sam Adriaensen
  \thanks{Department of Mathematics and Data Science, Vrije Universiteit Brussel, Pleinlaan 2, 1050 Elsene, Belgium. \href{mailto:sam.adriaensen@vub.be}{sam.adriaensen@vub.be}} 
  \thanks{Department of Mathematical Sciences, Worcester Polytechnic Institute, 100 Institute Road, 01609 Worcester, MA, US.}
 \and Tam\'as Sz\H onyi
  \thanks{HUN-REN-ELTE Geometric and Algebraic Combinatorics Research Group, P\'azm\'any P\'eter s\'et\'any 1/C, H-1117 Budapest, Hungary.
   \href{mailto:tamas.szonyi@ttk.elte.hu}{tamas.szonyi@ttk.elte.hu},
   \href{mailto:zsuzsa.weiner@gmail.com}{zsuzsa.weiner@gmail.com}} 
  \thanks{Institute of Mathematics, ELTE E\"otv\"os Lor\'and University, P\'azm\'any P\'eter s\'et\'any 1/C, H-1117 Budapest, Hungary,}
  \thanks{University of Primorska FAMNIT, Glagolja\v ska ulica 8, 6000 Koper, Slovenia.}
 \and Zsuzsa Weiner
  \footnotemark[3]
}

\date{}

\begin{document}

\maketitle

\begin{abstract}
 In this paper, we tie together two well studied topics related to finite Desarguesian affine and projective planes.
 The first topic concerns directions determined by a set, or even a multiset, of points in an affine plane.
 The second topic concerns the linear code generated by the incidence matrix of a projective plane.
 We show how a multiset determining only $k$ special directions, in a modular sense, gives rise to a codeword whose support can be covered by $k$ concurrent lines.
 The reverse operation of going from a codeword to a multiset of points is trickier, but we describe a possible strategy and show some fruitful applications.

 Given a multiset of affine points, we use a bound on the degree of its so-called projection function to yield lower bounds on the number of special directions, both in an ordinary and in a modular sense.
 
 In the codes related to projective planes of prime order $p$, there exists an odd codeword, whose support is covered by 3 concurrent lines, but which is not a linear combination of these 3 lines.
 We generalise this codeword to codewords whose support is contained in an arbitrary number of concurrent lines.
 In case $p$ is large enough, this allows us to extend the classification of codewords from weight at most $4p-22$ to weight at most $5p-36$.
\end{abstract}



\section{Introduction}

The goal of this article is to build a connection between two well-studied topics in finite geometry, namely (multi)sets determining few directions and small weight codewords of linear codes related to finite projective planes.
In addition to showing how the topics are connected, we will contribute to the study of both of them.
Despite both topics being two sides of the same coin, we have opted for an independent treatment where possible.
This should allow readers who are only interested in one of the topics to ignore the parts of the paper dealing with the other.

\bigskip

Throughout this article, $p$ will denote a prime number and $q=p^h$ a prime power.
Some more finite geometry notation is introduced in \Cref{Sec:Preliminaries}.
This should be standard to finite geometers, but the less initiated reader might want to consult the preliminaries before reading the introduction.

\subsection{Directions}

A set $S$ of points in $\ag(2,q)$ is said to \emph{determine a direction} $(d) \in \FF_q \cup \{\infty\}$ if there are two points in $S$ which are joined by a line with slope $d$.
There has been intense investigation into sets that determine few directions, see e.g.\ the seminal work of R\'edei \cite{lacunary}, Blokhuis, Ball, Brouwer, Storme, and Sz\H onyi \cite{BBBSS}, and Ball \cite{Ball03}.
Most attention has been paid to sets $S$ of size $q$, since the pigeonhole principle implies that any set of size greater than $q$ determines every direction.
However, Ghidelli \cite{Ghidelli} recently gave a sensible definition for directions associated to larger sets, see also the work of Kiss and Somlai \cite{Kiss:Somlai}.

\begin{df}
 \label{Df:SpecialDir}
 Given a set $S$ in $\ag(2,q)$, we say that $S$ is \emph{equidistributed} from direction $(d) \in \FF_q \cup \{\infty\}$ if every line with slope $d$ contains $\floor{|S|/q}$ or $\ceil{|S|/q}$ points of $S$.
 A direction from which $S$ is not equidistributed is called a \emph{special direction}.
\end{df}

Note that for a set $S$ whose size is at most $q$, a direction is determined if and only if it is special.
Ghidelli \cite{Ghidelli}, focusing on the case where $q$ is prime, proved the following result, generalising a result of R\'edei \cite{lacunary} and the second author \cite{Szonyi96, Szonyi99}.

\begin{res}[{\cite[Theorem 1.3]{Ghidelli}}]
 \label{Res:Ghidelli}
 Let $p$ be prime and $S$ a set of $np-r$ points in $\ag(2,p)$, with $1 \leq n \leq p$ and $0 \leq r < p$.
 Then $S$ is either contained in the union of $n$ lines or $S$ has at least $\ceil{\frac{p+n+2-r}{n+1}}$ special directions.
\end{res}

Ghidelli \cite[Problem 1.4]{Ghidelli} asked whether the above bound can be improved to $\ceil{\frac{p+3-r}2}$.
Kiss and Somlai \cite{Kiss:Somlai} investigated special directions determined by a set $S$ of points in $\ag(2,p)$, $p$ prime, in the case that $p$ divides $|S|$.
Note that in this case, a direction $(d)$ is equidistributed if and only if each line with slope $d$ contains exactly $|S|/p$ points of $S$.
They answered Ghidelli's question in the negative.
Given a prime number $p$, we can interpret the elements of $\FF_p$ as the integers from $0$ to $p-1$.
This defines in a natural way an ordering $<$ on the elements of $\FF_p$.

\begin{res}[{\cite[Theorem 1.1]{Kiss:Somlai}}]
 \label{Res:KissSomlai}
 Let $p > 2$ be prime.
 Consider the set of $\frac{p-1}2 p$ points
 \[
  S = \sett{ (x,y) \in \FF_p^2 }{ y < x }
 \]
 in $\ag(2,p)$.
 This set has exactly 3 special directions.
 Moreover, every set of $\ag(2,p)$ having exactly 3 special directions is equivalent up to affine transformation to either $S$ or its complement.
\end{res}

The above construction of the set $S$ is quite intriguing.
We can interpret the set $S$ as the set of points ``below'' the line $Y = X$.
This begs the question whether the construction can be generalised to yield sets determining e.g.\ $4$ special directions.
A first instinct might be to look at the set of points below a parabola in $\ag(2,p)$.
While such a set exhibits some interesting behaviour, it does not have any equidistributed directions, save perhaps when $p$ is very small.
This is addressed in a separate note \cite{Parabola}.
Then how does one generalise the construction of the set $S$ from \Cref{Res:KissSomlai}?
We will see later how $S$ arises from an \emph{odd codeword} of a linear code associated to the projective plane $\pg(2,p)$.
This allows us to give an alternative proof for the fact that $S$ has 3 special directions, see \Cref{Thm:KissSomlaiAgain}.
This link can be generalised, but we need a less restrictive notion than sets of points and their special directions.

As a first step, we generalise the framework of Kiss and Somlai \cite{Kiss:Somlai} from sets of points in $\ag(2,p)$ to multisets of points in $\ag(2,q)$.
Whenever we refer to the size of a multiset, sum over the elements of a multiset, etc., we always take into account the multiplicities, see \Cref{Sec:Preliminaries} for more details.
For the reader's convenience, we will denote multisets of points by $M$;
when we are speaking about a set, not a multiset, we will denote it as $S$.

\Cref{Df:SpecialDir} naturally extends to multisets of points in $\ag(2,q)$, and we will use the same terminology.
We emphasise that when we say that a line contains $k$ points of a multiset, we do take into account the multiplicities.
In addition, we will also use the following more relaxed definition.

\begin{df}
 Let $M$ be a multiset in $\ag(2,q)$.
 We say that the direction $(d) \in \fqi$ is \emph{modularly equidistributed}, or simply \emph{mod-equidistributed}, if every line with slope $d$ contains the same number of points of $M$ modulo $p$.
 If $(d)$ is not modularly equidistributed, we call it \emph{modularly special} or \emph{mod-special}.
\end{df}

We would like to point out that if $(d_1)$ and $(d_2)$ are mod-equidistributed directions of $M$, and every line with slope $d_i$ contains $r_i$ modulo $p$ points, this need not imply that $r_1 = r_2$.

\begin{df}
 \label{Df:Pr}
Consider a multiset $M$ of points in $\ag(2,q)$, with $q=p^h$ and $p$ prime.
For each direction $(d) \in \FF_q \cup \{\infty\}$, define the \emph{projection function} $\pr_{M,d}$ of $M$ from $(d)$ as the function $\FF_q \to \FF_p$, which maps $b \in \FF_q$ to the number of points of $M$ on the line $Y = dX+b$, or the line $X+b=0$ if $d = \infty$, reduced modulo $p$.
This function coincides with a unique polynomial in $\FF_q[X]$ of degree at most $q-1$, and we will interpret $\pr_{M,d}$ as this polynomial.
\end{df}

The projection function was used originally by Kiss and Somlai in the special case where $q$ is prime and $M$ is an ordinary set.
They proved the following proposition.

\begin{res}[{\cite[Proposition 3.1]{Kiss:Somlai}}]
 Let $S$ be a set of points in $\ag(2,p)$, with $p$ prime, determining $k \geq 2$ special directions.
 Then for every direction $(d)$, $\deg \pr_{S,d} \leq k-2$.
\end{res}

The original proof uses the so-called \emph{R\'edei polynomial}, which encodes the size of $S \cap \ell$ for a line $\ell$ as the multiplicity of certain roots.
We will prove that this proposition extends to our more general setting, by encoding intersection sizes in an additive instead of multiplicative way.
Although the use of the R\'edei polynomial extends to our setting, the additive approach avoids the use of more advanced results such as the Newton identities, and highlights the connection with coding theory.

\begin{thm}
 \label{Thm:Pr}
Let $M$ be a multiset of points in $\ag(2,q)$ determining $k \geq 2$ mod-special directions.
 Then for every direction $(d)$, $\deg \pr_{M,d} \leq k-2$.
\end{thm}

The degree bound has some interesting consequences.

\begin{thm}
 \label{Thm:MinModSpecialDir}
 Let $M$ be a multiset of points in $\ag(2,q)$, with $q=p^h$ and $p$ prime.
 Then $M$ has either 0 or at least $\frac qp + 2$ mod-special directions.
\end{thm}

\begin{thm}
 \label{Thm:MinSpecialDir}
 Let $M$ be a multiset of $nq$ points in $\ag(2,q)$, with $q=p^h$, $p$ prime, and $h > 1$.
 If $M$ is not the union of $n$ (not necessarily distinct) lines, then $M$ determines at least $\frac pn + 1$ special directions.
\end{thm}

We point out that the bound from \Cref{Thm:MinSpecialDir} is tight in case $h=2$ and $n=1$, as evidenced by taking $M$ to be an affine Baer subplane. 

Our result has some implications on \emph{multiple blocking sets of R\'edei type} as well.
We refer the interested reader to \Cref{Crl:Redei}.

\bigskip

Let us return to the question of how we can generalise \Cref{Res:KissSomlai}.
In \Cref{Sec:SetsToCode}, we will see how finding constructions of multisets in $\ag(2,q)$ with $k$ mod-special directions corresponds to finding explicit linear combinations of certain codewords in the code generated by the incidence matrix of $\pg(2,q)$.
Finding such explicit linear combinations is discussed in \Cref{Sec:LinComb}.
One consequence is that if $p$ is prime, then for each integer $k \in [3,p+1]$, there are multisets $M$ of points in $\ag(2,p)$ which determine exactly $k$ mod-special directions, but which are not unions of multisets with fewer than $k$ mod-special directions.
We refer to \Cref{Crl:MultisetNotUnion} for more details.

\subsection{Small weight codewords}

There are several codes related to subspaces of projective and affine spaces. 
Probably, the best known ones are the Reed-Muller codes.
The codes can be interpreted in both an algebraic way, related to multivariate polynomials over $\FF_2$, and in a geometric way, corresponding to subspaces in $\ag(n,2)$ or $\pg(n,2)$.
Delsarte, Goethals and MacWilliams \cite{delsartegoethalsmacwilliams} investigated the generalisations of both the algebraic and geometric description the Reed-Muller codes to larger finite fields.
They obtained important properties of the geometric generalisation by seeing them as subcodes of the algebraic generalisation.
The projective geometric generalisation is defined as follows.

\begin{df}
 Suppose that $p$ is prime and $q = p^h$.
 Let $\mp$ denote the set of points of $\pg(n,q)$.
 Given a $k$-space $\pi$ of $\pg(n,q)$, define its \emph{characteristic function} as
 \[
  \chi_{\pi}: \mp \to \FF_p: P \mapsto \begin{cases}
   1 & \text{if } P \in \pi, \\
   0 & \text{otherwise}.
  \end{cases}
 \]
 Let $\mc_k(n,q)$ denote the $\FF_p$-span of the characteristic functions of the $k$-spaces of $\pg(n,q)$.
 The codes $\mc_1(2,q)$ will simply be denoted as $\mc(2,q)$.
\end{df}

\begin{rmk}
 The code $\mc_k(n,q)$ is often defined in a different but equivalent way.
 It can be seen as the row space of the \emph{incidence matrix} of $k$-spaces and points of $\pg(n,q)$, when interpreted as a matrix over $\FF_p$.
 This perspective will be useful in \Cref{Sec:LinComb}.
\end{rmk}

The main questions are to determine the dimension and minimum distance of such codes.
The dimension of $\mc_k(n,q)$ was determined by Hamada \cite[Theorem 1]{hamada1968}.
Determining the minimum distance of a linear code is equivalent to determining its minimum weight.

\begin{df}
 Let $\mp$ be a set and $c:\mp \to \FF_q$ a function.
 The \emph{support} of $c$ is the set
 \[
  \supp(c) = \sett{P \in \mp}{c(P) \neq 0}
 \]
 and the \emph{weight} of $c$, denoted $\wt(c)$, is the size of $\supp(c)$.
\end{df}

The minimum weight of the codes $\mc_k(n,q)$ has been known for a long time.

\begin{res}[{\cite[\S 5.2.1]{delsartegoethalsmacwilliams}}]
 \label{Res:MinWt}
 The minimum weight of $\mc_k(n,q)$ is $\frac{q^{k+1}-1}{q-1}$, and the only codewords attaining this weight are the characteristic functions of the $k$-spaces and their scalar multiples.
\end{res}

More generally, a great deal of effort has been devoted to characterising codewords of $\mc_k(n,q)$ of relatively small weight.
We refer the reader to the excellent survey of
Lavrauw, Storme, and Van de Voorde \cite{Lavrauw:Storme:VandeVoorde:10}, and to the paper of the first author and Denaux \cite{Adriaensen:Denaux:24} for an overview of more recent results.

In the present paper, only the codes $\mc(2,q)$ and their small weight codewords are of interest to us.
A simple way to make a small weight codeword is to take a linear combination $c$ of the characteristic functions of a relatively small set of lines $\ell_1, \dots, \ell_n$.
In such a case, we will simply say that $c$ is a \emph{linear combination of the lines $\ell_1, \dots, \ell_n$}.
The last two authors of this paper proved the following result.

\begin{res}[{\cite[Theorem 4.3]{Szonyi:Weiner:14}}]
 \label{Res:SzWNonPrime}
 Suppose that $p$ is prime and $q = p^h \geq 32$ with $h \geq 2$.
 If $c \in \mc(2,q)$ and
 \[
  \wt(c) < \begin{cases}
   \frac{(p-1)(p-4)(p^2+1)}{2p-1} & \text{if } h = 2, \\
   \floor{\sqrt q + 1}\ceil{q - \sqrt q + 1} & \text{if } h > 2,
  \end{cases}
 \]
 then $c$ is a linear combination of $\ceil{\frac{\wt(c)}{q+1}}$ lines.
\end{res}

The above bound on $\wt(c)$ is known to be tight when $h > 2$ is even, as the characteristic functions of \emph{Hermitian unitals} are codewords of $\mc(2,q)$, see \cite{blokhuisbrouwerwilbrink}, which have weight $(\sqrt q + 1)(q - \sqrt q + 1)$, but cannot be written as the linear combination of only few lines.

However, the above result no longer holds in the case that $q=p$ is prime.

\begin{ex}
 \label{Ex:OddOn3}
 Let $p$ be prime and $\mp$ be the set of points of $\pg(2,p)$.
Consider the function
\[
 c: \mp \to \FF_p: P \mapsto \begin{cases}
  z & \text{if } P = (1,0,z) \text{ or } P = (0,1,z) \text{ or } P = (1,1,-z), \\
  0 & \text{otherwise}.
 \end{cases}
\]
Note that $\wt(c) = 3(p-1)$ and $\supp(c)$ is contained in the union of the 3 lines $X=0$, $Y=0$, and $X=Y$, which are concurrent at the point $(0,0,1)$.
It was proved by Bagchi \cite[Theorem 5.2]{Bagchi} and independently by De Boeck and Vandendriessche \cite[Example 10.3.4]{DeBoeckPhd} that $c$ is a codeword of $\mc(2,p)$.
Moreover, if $p \geq 5$, then $c$ is not a linear combination of the lines $X=0$, $Y=0$, and $X=Y$.
\end{ex}

\begin{df}
 \label{Df:OddCodeword}
 We call a codeword $c \in \mc(2,q)$ an \emph{odd codeword on $n$ lines} if the following conditions hold:
 \begin{enumerate}
  \item There exist $n$ lines $\ell_1, \dots, \ell_n$ concurrent at the point $P$, such that $\supp(c)$ is contained in the union of these lines.
  \item For none of the lines $\ell_i$, the codeword $c$ is constant on the points of $\ell_i \setminus \{P\}$.
 \end{enumerate}
\end{df}

Informally, the second condition assures that $c$ is not the sum of an odd codeword on fewer than $n$ lines through $P$ and a linear combination of some more lines through $P$.

We saw an odd codeword on 3 lines in \Cref{Ex:OddOn3}.
In \Cref{Thm:ConcurrentCodeword}, we will describe for general $n$ how to construct all odd codewords on $n$ lines.
This theorem is especially useful for the codes $\mc(2,p)$ with $p$ prime.


\bigskip

Similar to \Cref{Res:SzWNonPrime}, small weight codewords of $\mc(2,p)$ with $p$ prime have been characterised, albeit with a more strict bound on the weight.


\begin{res}[{\cite[Theorem 4.8, Corollary 4.10]{Szonyi:Weiner:14}}]
 \label{Res:SzWPrime}
 Suppose that $p \geq 19$ is prime and that $c \in \mc(2,p)$ with $\wt(c) \leq \max\{3p+1, 4p-22\}$.
 Then $c$ is either projectively equivalent to \Cref{Ex:OddOn3} plus a linear combination of the 3 concurrent lines covering its support, or it is the linear combination of at most 3 lines.
\end{res}

We will extend this characterisation to higher weights.

\begin{thm}
 \label{Thm:SmallWeight}
Suppose that $p \geq 37$ is prime and that $c \in \mc(2,p)$ with
\[
 \wt(c) \leq \begin{cases}
  4p+3 & \text{if } p \leq 47, \\
  5p-36 & \text{if } p \geq 53.
 \end{cases}
\]
 Then either
 \begin{enumerate}
  \item $c$ is a linear combination of at most 4 lines,
  \item $c$ is a linear combination of an odd codeword on 3 lines and another line,
  \item or $c$ is an odd codeword on 4 lines.
 \end{enumerate}
\end{thm}

\subsection{From multisets to codewords and back}
 \label{Sec:SetsToCode}

There is a correspondence between multisets of points in $\ag(2,q)$ with $k$ mod-special directions and explicit ways to write odd codewords of $\mc(2,q)$ on $k$ lines as linear combinations of lines.
We start by going from a multiset to an odd codeword.

\begin{nota}
 We will write coordinates of points as $(d,-1,b)$ where $d \in \fqi$.
 In case $d = \infty$, this should be read as $(1,0,b)$.
 Similarly, $[\infty,-1,b]$ denotes the line $[1,0,b]$, and the equation $Y = \infty X + b$ should be read as $0 = X + b$
\end{nota}

Recall the projection function defined in \Cref{Df:Pr}.

\begin{prop}
 \label{Prop:SetToCode}
 Let $M$ be a multiset of points in $\ag(2,q)$.
 Let $D, \overline D \subseteq \fqi$ denote the sets of mod-special and mod-equidistributed directions of $M$ respectively.
 For each $(\overline d) \in \overline D$, let $r_{\overline d}$ be the integer in $\{0,\dots,p-1\}$ such that every line with slope $\overline d$ contains $r_{\overline d}$ points of $M$ modulo $p$.
 Define
 \[
  c_M = \sum_{(x,y) \in M} \chi_{[x,y,1]} - \sum_{\overline d \in \overline D} r_{\overline d} \chi_{[1,\overline d,0]}.
\]
 Then for any $d \in \fqi$ and $b \in \FF_q$,
 \begin{equation}
   \label{Eq:SetToCode}
  c_M(d,-1,b) = \begin{cases}
   \pr_{M,d}(b) & \text{if } d \in D, \\
   0 & \text{otherwise.}
  \end{cases}
 \end{equation}
 Therefore, $c_M$ is an odd codeword on the $|D|$ lines $X + dY = 0$ with $d \in D$, concurrent at $(0,0,1)$.
\end{prop}

\begin{proof}
 First consider the codeword $c_M' = \sum_{(x,y) \in M} \chi_{[x,y,1]}$.
 Choose $d \in \fqi$ and $b \in \FF_q$.
 Then
 \[
  c_M'(d,-1,b) = |\sett{(x,y) \in M}{y = dx + b}| = \pr_{M,d}(b).
 \]
 By subtracting $\sum_{\overline d \in \overline D} r_{\overline d} \chi_{[1,\overline d,0]}$, we change this value to 0 if $(d)$ is mod-equidistributed, and leave it unchanged if $(d)$ is mod-special.
 This proves \Cref{Eq:SetToCode}.

 It also follows that $\supp(c_M)$ is contained in the union of the lines $X+dY = 0$ with $d \in D$.
 The fact that $c_M$ is an odd codeword follows from $c_M(d,-1,b) = \pr_{M,d}(b)$ not being a constant function of $b$ in case that $d \in D$.
\end{proof}

Conversely, every way to write an odd codeword on $k$ lines as an explicit linear combination of lines gives us a multiset of points in $\ag(2,q)$ with $k$ mod-special directions.

\begin{nota}
 \label{Not:Nu}
 For a prime number $p$, let $\nu: \FF_p \to \QQ$ be the function mapping an element $x \in \FF_p$ to the corresponding integer in $\{0, \dots, p-1\}$.
\end{nota}

\begin{prop}
 \label{Prop:CodeToSet}
 Suppose that $c \in \mc(2,q)$ is an odd codeword on the $k$ concurrent lines $Y = d X$ with $d \in D \subseteq \fqi$.
 Suppose that
 \[
  c = \sum_{\ell} \alpha_{\ell} \chi_\ell
 \]
 is an explicit way to write $c$ as a linear combination of lines.
 Construct the multiset $M$ of points in $\ag(2,q)$ where the point $(a,b)$ has multiplicity $\nu(\alpha_{[a,b,1]})$.
 Then
 \begin{enumerate}
  \item $D$ is the set of mod-special directions of $M$,
  \item for every direction $(\overline d) \in (\fqi) \setminus D$, all lines with slope $\overline d$ intersect $M$ in $\nu(-\alpha_{[1,\overline d, 0]})$ points modulo $p$.
 \end{enumerate}
\end{prop}

\begin{proof}
 Take $d \in \fqi$ and $b \in \FF_q$.
 Let $\ell$ be the line with equation $Y = d X + b$.
 Then
 \[
  c(d,-1,b) = \sum_{\substack{(x,y) \in \FF_q^2 \\ y = d x + b}} \alpha_{[x,y,1]} + \alpha_{[1,d,0]}
  = |M \cap \ell| + \alpha_{[1,d,0]}.
 \]
 Therefore, $(d)$ is a mod-special direction of $M$ if and only if $c(d,-1,b)$ is not a constant function of $b$.
 Since $c$ is an odd codeword, this is equivalent to $(d) \in D$.
 Moreover, if $(\overline d) \notin D$, then $c(\overline d,-1,b) = 0$ for all $b$, which means that all lines with slope $\overline d$ intersect $M$ in $\nu(-\alpha_{[1,\overline d, 0]})$ points modulo $p$.
\end{proof}


While each multiset $M$ gives rise to a unique odd codeword $c$, every odd codeword corresponds to many linear combinations of lines, and therefore to many multisets.
It is in general not easy to find the linear combinations corresponding to $c$.
We will explore how we can overcome this hurdle in \Cref{Sec:LinComb}.

\subsection{Structure of the paper}

\Cref{Sec:Preliminaries} contains the preliminaries of the paper.
The purpose of this section is to provide clarity regarding the notation used throughout the paper.

\Cref{Sec:ModSpecial} deals exclusively with directions determined by multisets of points.
We prove results related to special and mod-special directions of multisets.
Most notably, we prove Theorems \ref{Thm:Pr}, \ref{Thm:MinModSpecialDir}, and \ref{Thm:MinSpecialDir}.
In \Cref{Crl:Redei}, we discuss the implications of our results to minimal weighted multiple blocking sets of R\'edei type.

Sections \ref{Sec:Concurrent} and \ref{Sec:Small} deal exclusively with the codes $\mc(2,q)$.
In \Cref{Sec:Concurrent}, we state a powerful result by Delsarte, Goethals, and MacWilliams (\Cref{Res:DGM}) that describes the codewords of $\mc_k(n,q)$ as certain multivariate polynomials.
We use this to derive some properties concerning codewords of $\mc(2,q)$ whose support is contained in concurrent lines.
In \Cref{Sec:Small}, we prove \Cref{Thm:SmallWeight}, which extends the classification of small weight codewords in codes from projective planes of prime order.

In \Cref{Sec:LinComb}, we explore how to write an odd codeword as an explicit linear combination of lines.
As a result, we recover the set described in \Cref{Res:KissSomlai}.

In \Cref{Sec:Conclusion}, we conclude the paper by discussing interesting questions that remain unanswered.

\section{Preliminaries}
 \label{Sec:Preliminaries}

 Throughout this paper, $p$ denotes a prime number and $q = p^h$ a power of $p$.
 The finite field of order $q$ will be denoted as $\FF_q$.
 We denote the $n$-dimensional Desarguesian affine and projective geometry over $\FF_q$ by $\ag(n,q)$ and $\pg(n,q)$ respectively.
 This means that $\pg(n,q)$ consists of the subspaces of $\FF_q^{n+1}$, and the projective dimension of a subspace is one less than its vector space dimension.
 We use $\vspan{\cdot}$ to denote the linear span.
 If $x \in \FF_q^{n+1}$ is a non-zero vector, then $P = \vspan x$ is a point of $\pg(n,q)$.
 By slight abuse of notation, we will simply write this as $P = x$.
 The affine space $\ag(n,q)$ is obtained from $\pg(n,q)$ by choosing a hyperplane $\Pi_\infty$ and throwing away all subspaces which are completely contained in $\Pi_\infty$.

 We can also represent these geometries in another way, which we will only describe in case $n=2$.
 We can represent the points of $\ag(2,q)$ as the vectors of $\FF_q^2$.
 The lines of $\ag(2,q)$ are defined by an equation of the form
 \[
  a Y = d X + b,
 \]
 with $(a,d,b) \in \FF_q^3 \setminus \{\zero\}$.
 We call $d/a \in \fqi$ the \emph{slope} of the corresponding line.
 We can complete the affine plane to a projective plane as follows.
 For each element $d \in \fqi$, add a point $(d)$.
 Add a line $\ell_\infty$ containing exactly the points $(d)$ with $d \in \fqi$.
 The point $(d)$ lies exactly on the line $\ell_\infty$ and the lines with slope $d$.
 We call $(d)$ a \emph{direction}.
 We will identify the point $(x,y) \in \FF_q^2$ of $\ag(2,q)$ with the point $(x,y,1)$ of $\pg(2,q)$.
 The points $(d)$ with $d \in \FF_q$ and $(\infty)$ are identified with $(1,d,0)$ and $(0,1,0)$ respectively.

 We will use the notation $[a,b,c]$ for the line of $\pg(2,q)$ with equation $aX + bY + cZ = 0$.
 In particular, $\chi_{[a,b,c]}$ denotes the characteristic function of this line.


 \bigskip

 Multisets will play an important role in this paper.
 We opt for a compact and informal notation, but for the sake of clarity, we will give a formal definition here.
 A multiset $M$ can be seen as a pair $(S,\mu)$, with $S$ a set and $\mu: S \to \NN$ a function assigning a positive multiplicity to each element of $S$.
 If $f$ maps the elements of $S$ into some additive or multiplicative structure, then $\sum_{x \in M} f(x)$ and $\prod_{x \in M} f(x)$ take into account the multiplicities.
 Formally,
 \begin{align*}
  \sum_{x \in M} f(x) = \sum_{x \in S} \mu(x) f(x), &&
  \prod_{x \in M} f(x) = \prod_{x \in S} f(x)^{\mu(x)}.
 \end{align*}
 Similarly, if $T$ is a set, then $|T \cap M|$ takes into account multiplicities, or formally
 \[
  |T \cap M| = \sum_{x \in T \cap S} \mu(x).
 \]
 In particular, the size of $M$ is $\sum_{x \in S} \mu(x)$.
 
 We will sometimes define a multiset as a union of some sets $S_1, \dots, S_n$.
 This means that $M = (\bigcup_{i=1}^n S_i, \mu)$ with
 \[
  \mu: \bigcup_{i=1}^n S_i \to \NN: x \mapsto |\sett{i \in \{1, \dots, n\}}{x \in S_i}|.
 \]

\section{Multisets in AG\texorpdfstring{$\boldsymbol{(2,q)}$}{(2,q)} and their directions}
 \label{Sec:ModSpecial}

\subsection{Multisets admitting at most 2 special directions}

In this section, we investigate multisets $M$ of points in $\ag(2,q)$ whose size is divisible by $q$, and which have few special directions.
First, we observe the trivial case of having no special directions.

\begin{lm}
\label{0_spec_dir}
  If $M$ is a multiset in $\ag(2,q)$ admitting no special direction, then every point of $\ag(2,q)$ has the same multiplicity in $M$. \qed
\end{lm}

The trivial example for multisets of size $nq$ admitting one special direction is the union of $n$ not necessarily distinct, parallel lines. Furthermore, it is easy to see that the converse holds as well.

\begin{prop}
\label{1_spec_dir}
  If $M$ is a multiset of size $nq$ in $\ag(2,q)$ admitting exactly one special direction, then $M$ must be the union of $n$ not necessarily distinct parallel lines.
\end{prop}

\begin{proof}
Let $P$ be a point of $M$ with multiplicity $m$. Counting the points of $M$ on the lines through $P$, we see that the line with the slope of the special direction must contain $mq$ points of $M$. This also means that every point on this affine line must have multiplicity exactly $m$.
\end{proof}

There is another simple construction where the number of special directions is at most $n$.

\begin{ex}
The union of $n$ lines in $\ag(2,q)$ with $k$ different slopes is a multiset of size $nq$ admitting at most $k$ special directions.
\end{ex}

Kiss and Somlai \cite[\S 2]{Kiss:Somlai} gave a simple combinatorial proof that the only sets of size $np$ in $\ag(2,p)$ determining only 2 special directions are unions of parallel lines, which also follows from \Cref{Res:Ghidelli}.
Their argument generalises to multisets in $\ag(2,q)$ as follows.

\begin{prop}
\label{2_spec_dirs}
 Suppose that a multiset $M$ of size $nq$ determines exactly $2$ special directions. Then $M$ is the union of lines through these special directions.
\end{prop}

\begin{proof}
 We may suppose that the special directions are $(0)$ and $(\infty)$.
 Let $\mu(x,y)$ denote the multiplicity of $M$ at the point $(x,y)$.
 Suppose that the lines $X = x$ and $Y = y$ intersect $M$ in $a_x$ and $b_y$ points, respectively.
 By counting the lines through a point $(x,y)$, we see that
 \[
  (q-1) n + a_x + b_y = qn + q \mu(x,y).
 \]
 In other words,
 \[
  \mu(x,y) = \frac{a_x + b_y - n}{q}.
 \]
 This means that $a_x + b_y \equiv n \pmod{q}$ for all $x$ and $y$.
 Fixing $y$ and varying $x$, we see that there exist integers $r_x$ such that $a_x = a + q r_x$ for some $a$ with $0 \leq a < q$.
 Likewise, $b_y = b + q s_y$ for some integers $s_y$ and some $b$ with $0 \leq b < q$.
 Moreover, $a + b \equiv n \pmod q$, hence $a+b = n - qt$ for some integer $t$.

 Plugging this into the equation, we see that
 \[
  \mu(x,y) = r_x + s_y - t.
 \]
Take $x$ with $r_x$ minimal and $y$ with $s_y$ minimal.
Then $\mu(x,y) \geq 0$, hence we can write $t = t_1 + t_2$ with $0 \leq t_1 \leq r_x$ for each $x$ and $0 \leq t_2 \leq r_y$ for each $y$.
This means that we can obtain $M$ as the union of $r_x - t_1$ times the line $X = x$ and $s_y - t_2$ times the line $Y = y$ for all $x$ and $y$.
\end{proof}

The next corollary follows from the proposition above and from Proposition \ref{1_spec_dir} and \ref{0_spec_dir}.

\begin{crl}
Multisets of size $nq$ in AG$(2,q)$ admitting fewer than $3$ special directions are unions of $n$ parallel lines from at most $2$ different parallel classes.    
\end{crl}

\subsection{Proofs of Theorems \texorpdfstring{\ref{Thm:Pr}, \ref{Thm:MinModSpecialDir}, \ref{Thm:MinSpecialDir}}{}}
 \label{Sec:ProofsOfTheorems}

\begin{prop}
 \label{Prop:AddRedPol}
 Consider a multiset $M$ of points in $\ag(2,q)$.
 Suppose that $(\infty)$ is a mod-special direction and let $E \subseteq \FF_q$ denote the set of mod-equidistributed directions of $M$.
 For each $e \in E$, $\pr_{M,e}(b)$ is a constant function, and we denote this constant as $r_e$.
 Define the polynomial
 \[
 F_M(T,D,B) = \sum_{(x,y) \in M} 1 - (yT - (Dx + B))^{q-1} - \sum_{e \in E} r_e (1-(D-eT)^{q-1}).
\]
Then for $d, b \in \FF_q$
\begin{align*}
 F_M(1,d,b) &= \begin{cases}
  \pr_{M,d}(b) & \text{if } d \notin E, \\
  0 & \text{if } d \in E
 \end{cases} \\
 F_M(0,1,b) &= \pr_{M,\infty}(b).
\end{align*}
\end{prop}

\begin{proof}
 Take a vector $(t,d,b)$ with $(t,d) \neq \zero$.
 Then the polynomial $1 - (Yt - (dX + b))^{q-1}$ in $X$ and $Y$ is $1$ on points of the line $tY = dX+b$ and 0 otherwise. Hence the first term of $F_M$ counts the number of points of $M$ on the line $tY = dX +b$ modulo $p$.  Likewise, $(1-(D-dT)^{q-1})$ is 1 if $D = dT$ and 0 otherwise and so the result follows. 
\end{proof}

\noindent {\bf \Cref{Thm:Pr}.} {\it 
 Let $M$ be a multiset of points in $\ag(2,q)$ determining $k \geq 2$ mod-special directions.
 Then for every direction $(d)$, $\deg \pr_{M,d}(B) \leq k-2$.
}

\begin{proof}
The theorem holds trivially for mod-equidistributed directions. Let $(d')$ be a mod-special direction. Note that an affine transformation does not alter the degree of the projection functions. Thus, since $k \geq 2$, we may assume that $(\infty)$ is mod-special and $d' \neq \infty$. In \Cref{Prop:AddRedPol}, substitute $1$ for $T$ and consider the two-variable polynomial
\[
F_M(1, D, B) = \sum_{i=0}^{q-1} f_i(D) B^{q-1-i},
\]
where $\deg f_i(D) \leq i$. By \Cref{Prop:AddRedPol}, we have $F_M(1, d, B) \equiv 0$ when $(d)$ is a mod-equidistributed direction, since $\deg_B F_M \leq q-1$. Note that there are $q+1-k$ such directions, implying $f_i(D) \equiv 0$ for $i < q+1-k$ and hence $\deg F_M(1, d', B) \leq k - 2$. The result now follows from \Cref{Prop:AddRedPol}.
\end{proof}

\noindent {\bf \Cref{Thm:MinModSpecialDir}.}  {\it
 Let $M$ be a multiset of points in $\ag(2,q)$, with $q=p^h$ and $p$ prime.
 Then $M$ has either 0 or at least $\frac qp + 2$ mod-special directions.
}

\begin{proof}
 Suppose that $(d)$ is a mod-special direction.
 Define the polynomial $f(B) = \pr_{M,d}(B)^p - \pr_{M,d}(B)$.
 Then $\deg f = p \deg \pr_{M,d} > 0$.
 Since $\pr_{M,d}$ only takes values in $\FF_p$, $f$ is zero everywhere and therefore its degree is at least $q$.
 Thus, $\deg \pr_{M,d} \geq q/p$. Note that counting the points of $M$ on the lines of through a point in $\ag(2,q)$ yields that $M$ cannot have exactly $1$ mod-special direction and hence
 the theorem follows from \Cref{Thm:Pr}.
\end{proof}

\begin{lm}
\label{less than p points}
Let $M$ be a multiset of size $nq$ in $\ag(2,q)$, $q=p^h$, $h > 1$. Assume that $M$ admits $k \leq \frac{p}{n}$ special directions.
Then any line whose slope is a special direction and which passes through a point not in $M$ contains fewer than $p$ points of $M$. 
\end{lm}

\begin{proof}
Let $P$ be a point not in $M$ and let us count the number of points of $M$ on the lines through $P$. 
Each line whose slope is an equidistributed direction contains exactly $n$ points.
Hence, on the lines through $P$ with slopes of a special direction, we see a total of $qn - (q + 1 - k)n = (k - 1)n$ points. By the assumption of the lemma, this is smaller than $p$.   
\end{proof}

\noindent{\bf \Cref{Thm:MinSpecialDir}.} {\it
 Let $M$ be a multiset of $nq$ points in $\ag(2,q)$, with $q=p^h$, $p$ prime, and $h > 1$.
 If $M$ is not the union of $n$ (not necessarily distinct) lines, then $M$ determines at least $\frac pn + 1$ special directions.
}

\begin{proof}
Suppose that $M$ determines at most $\frac pn$ special directions.
Since $\frac{p}{n} < \frac{q}{p} + 1$ if $h> 1$, \Cref{Thm:MinModSpecialDir} implies that every direction is mod-equidistributed.
Now suppose that $(d)$ is a special direction, and that every line with slope $d$ meets $M$ in $r$ points modulo $p$, with $0 \leq r < p$.
Let $\ell_b$ denote the line with equation $Y = dX + b$.
For each $b \in \FF_q$, let $m_b$ be the minimum multiplicity with respect to $M$ of the points on the line $\ell_b$.
Let $M'$ be the multiset obtained from $M$ by reducing the multiplicity of each point on $\ell_b$ by $m_b$ for each $b \in \FF_q$.
On each line $\ell_b$, we reduced the total multiplicities by $q m_b$, so $\ell_b$ still meets $M$ in $r$ points modulo $p$.
If $\ell$ is any line with slope different than $d$, then the sum of the multiplicities of the points on $\ell$ is reduced by $\sum_{b \in \FF_q} m_b$.
Therefore, any direction $(d') \neq (d)$ is special for $M'$ if and only if it is special for $M$.
Hence, $M'$ also has at most $\frac pn$ special directions.
Note that now every line with slope $d$ contains a point outside of $M'$. If $(d)$ was a special direction of $M'$, then by \Cref{less than p points}, every line with slope $d$ would contain fewer than $p$ points of $M'$. 
But then every line with slope $d$ must contain exactly $r$ points, which means that $M'$ is equidistributed from $(d)$.

Repeating this reduction process for every special direction of $M$ one by one, proves that $M$ is a union of lines.
\end{proof}

The above results have some consequences for weighted multiple blocking sets in PG$(2,q)$, $q=p^h$.
A \emph{weighted $n$-fold blocking set} of $\pg(2,q)$ is a multiset $M$ of points which intersects every line in at least $n$ points.
A point $P$ of $M$ is \emph{essential} if reducing the multiplicity of $P$ by $1$ results in a multiset that is no longer an $n$-fold weighted blocking set.
This is equivalent to the existence of a line through $P$ intersecting $M$ in exactly $n$ points.
We call $M$ \emph{minimal} if all of its points are essential.
We call a weighted multiple $n$-fold blocking set $M$ \emph{R\'edei type} if there exists a line $\ell$ such that the number of points of $M$ not lying on $\ell$ is $nq$.
The line $\ell$ is called a \emph{R\'edei line} of $M$.

\begin{crl}
 \label{Crl:Redei}
 Let $M$ be a minimal weighted $n$-fold R\'edei type blocking set in PG$(2,q)$, $q=p^h$, $h \geq 1$ and $n<q$, and let $\ell$ be a R\'edei line of $M$. 
\begin{enumerate}
 \item \label{Crl:Redei:1}
 If $\ell$ meets $M$ in at most $\frac{q}{p} + 1$ different points, then $M$ intersects every line in $n \mod p$ points.
 \item \label{Crl:Redei:2} If $h > 1$ and $\ell$ meets $M$ in fewer than $\frac{p}{n} + 1$ different points, then $M$ is the union of $n$ not necessarily distinct lines.
\end{enumerate}
\end{crl}

\begin{proof}
(1)
Suppose that $\ell = \ell_\infty$, i.e.\ the line with equation $Z = 0$.
Note that the lines containing no point of $M \cap \ell$ must intersect $M$ in exactly $n$ points (hence in $n$ modulo $p$ points). 
Let $P$ be point of $M$ on $\ell_\infty$.
Then by \Cref{Thm:MinModSpecialDir} there exists a value $r$ so that every affine line through $P$ meets $M$ in $r \mod p$ points. 

First assume that $\ell$ intersects $M$ in exactly $n$ points. Then every line containing a point outside $M$ must contain exactly $n$ points. Since $n < q$, we have an affine point $Q$ which is not in $M$ and so the line $PQ$ contains exactly $n$ points. It follows that $P$ must have multiplicity $n-r \mod p$ and so every line through $P$ intersects $M$ in $n \mod p$ points.

Now assume that $\ell$ contains more than $n$ points from $M$. Since $P$ is essential to $M$, there must be an affine line containing exactly $n$ points and so the proof for affine lines follows similarly as before. Counting the number of points through an affine point not in $M$ we see that $|M|$ is $n \mod p$. 
So repeating this for a point of $\ell_\infty$ which is not in $M$, we get that $\ell_\infty$ intersects $M$ in $n \mod p$ points too.

\bigskip

(2) This follows immediately from \Cref{Thm:MinSpecialDir}.
 \end{proof}

\begin{rmk}
Similar results were obtained in \cite[Theorem 2.13, Proposition 2.15]{FSSzW}.
But in \Cref{Crl:Redei}, (\ref{Crl:Redei:1}) also applies for relatively large $n$, and in (\ref{Crl:Redei:2}) we have a different condition to ensure that $M$ is the weighted sum of lines.     
\end{rmk}

\section{Odd codewords}
 \label{Sec:Concurrent}

In this section, we have a look at odd codewords as defined in \Cref{Df:OddCodeword}.

Delsarte, Goethals, and MacWilliams \cite{delsartegoethalsmacwilliams} described a useful way of recognising the codewords of $\mc_k(n,q)$ by expressing them as polynomials.
Let $\mp$ denote the set of points of $\pg(n,q)$.
To each function $c: \mp \to \FF_p$ we can assign a unique polynomial $F_c \in \FF_q[X_0, \dots, X_n]$ such that
\begin{enumerate}
 \item for each point $P$ of $\pg(n,q)$, every non-zero coordinate vector $(x_0, \dots, x_n)$ of $P$ satisfies $F_c(x_0, \dots, x_n) = c(P)$,
 \item the degree of $F_c$ in each variable is at most $q-1$,
 \item the total degree of $F_c$ is at most $n(q-1)$.
\end{enumerate}
Indeed, as a function $\FF_q^{n+1} \to \FF_p \subseteq \FF_q$, $F_c$ is uniquely determined except for its value in $\zero$.
Every function $\FF_q^{n+1} \to \FF_q$ can uniquely be represented as a multivariate polynomial with degree at most $q-1$ in each variable.
This means that the only possible term in $F_c$ of degree $(n+1)(q-1)$ is $X_0^{q-1} \cdot \ldots \cdot X_n^{q-1}$.
We can make sure that this monomial does not occur in $F_c$ by adding a scalar multiple of $(X_0^{q-1}-1) \cdot \ldots \cdot (X_n^{q-1}-1)$, since this only changes the value of the function in $\zero$.
Note that $F_c$ should take the same value on non-zero linearly dependent vectors, which is equivalent to each monomial occurring in $F_c$ having as total degree a multiple of $q-1$.
Hence, if we ensure that $\deg F_c < (n+1)(q-1)$, then $\deg F_c \leq n(q-1)$.

\begin{res}[{\cite[Theorem 5.2.3]{delsartegoethalsmacwilliams}}]
 \label{Res:DGM}
 Let $c$ be a function from the set of points of $\pg(n,q)$ to $\FF_p$.
 Then $c \in \mc_k(n,q)$ if and only if the total degree of $F_c$ is at most $(n-k)(q-1)$.
\end{res}

Recall the odd codewords from \Cref{Df:OddCodeword}.
So far, we have only seen an example of an odd codeword on 3 lines.
Now we will show the existence of odd codewords on $n > 3$ lines, which do not arise in a trivial way as linear combinations of odd codewords on fewer than $n$ lines.

\begin{thm} 
 \label{Thm:ConcurrentCodeword}
 Consider a set $D \subseteq \FF_q$.
 Suppose that $F \in \FF_q[X,Y,Z]$ has degree at most $q-1$ in each variable.
 Then $F$ is equal to $F_c$ for some codeword $c \in \mc(2,q)$ whose support is contained in the union of the lines $X=0$ and $Y = d X$ with $d \in D$, if and only if $F$ only takes values in $\FF_p$ and $F$ is of the form
 \begin{equation}
  \label{Eq:DGM}
  F(X,Y,Z) = G(X,Y,Z) \prod_{\overline d \in \FF_q \setminus D} (Y - \overline d X) + F(0,0,0) (1 - X^{q-1}),
 \end{equation}
 with $G$ a homogeneous polynomial of degree $|D| - 1$ or the zero polynomial.
\end{thm}

\begin{proof}
 Denote $\FF_q \setminus D$ as $\overline D$.
 Suppose that $F \in \FF_q[X,Y,Z]$ has degree at most $q-1$ in each variable and only takes values in $\FF_p$.
 By \Cref{Res:DGM}, $F = F_c$ for some codeword $c \in \mc(2,q)$ if and only if every term in $F$ has degree either $0$ or $q-1$.
 
 First suppose that $F$ is of the form written in \Cref{Eq:DGM}.
 It follows immediately that $F = F_c$  for some $c \in \mc(2,q)$.
 Moreover, $F(x,y,z) = 0$ whenever $x \neq 0$ and $y = \overline d x$ for some $\overline d \in \overline D$.
 Hence, $\supp(c)$ is contained in the aforementioned union of lines.

 Conversely, suppose that $F = F_c$ with $c \in \mc(2,q)$ a codeword whose support is contained in the aforementioned union of lines.
 Then all monomials in $F$ have degree either 0 or $q-1$.
 Write $F^*(X,Y,Z) = F(X,Y,Z) - F(0,0,0)(1 - X^{q-1})$.
 Then $F^*$ is homogeneous of degree $q-1$, as the constant coefficients of $F$ and $F(0,0,0)(1-X^{q-1})$ cancel out.
 If $\overline d \in \overline D$, then $F^*(1,\overline d,z) = c(1,\overline d,z) - 0 = 0$ for all $z$.
 Since the degree of $F^*(1,\overline d,Z)$ as polynomial of $Z$ is smaller than $q$, it must be the zero polynomial.
 Therefore, $Y-\overline d$ must divide $F^*(1,Y,Z)$.
 Since $F^*$ is homogeneous, this implies that $Y-\overline d X$ divides $F^*(X,Y,Z)$.
 The theorem follows.
\end{proof}

The above theorem is very valuable for the construction of odd codewords and proving non-existence.

\begin{ex}
 \label{Ex:OddOn4}
 In \Cref{Thm:ConcurrentCodeword}, let us choose $q = p > 2$ is prime, $D = \{-1,0,1\}$, $G = 2Z^2$, and $F(0,0,0) = 0$.
 Write $H(X,Y) = \prod_{\overline d \in \FF_q \setminus D} Y - \overline d X$.
 For general $q$ and $D$, using that the product of the elements of $\FF_q^*$ equals $-1$, it holds that
 \begin{align*}
  H(1,d) = \begin{cases}
   0 & \text{if } d \notin D, \\
   -\left( \prod_{d' \in D \setminus \{d\}} d - d' \right)^{-1} & \text{if } d \in D, 
  \end{cases} &&
   H(0,1) = 1.
 \end{align*}
 In our case, this means that
 \begin{align*}
  H(0,1) = H(1,0) = 1, && H(1,1) = H(1,-1) = -1/2.
 \end{align*}
 We find the odd codeword $c$ on 4 lines given by
 \[
  c(P) = \begin{cases}
   2 z^2 & \text{if $P = (1,0,z)$ or $P=(0,1,z)$}, \\
   - z^2 & \text{if $P = (1,1,z)$ or $P = (1,-1,z)$}, \\
   0 & \text{otherwise}.
  \end{cases}
 \]
\end{ex}

\begin{crl}
 \label{Crl:q/p+2}
 Consider $q = p^h$ with $p$ prime.
 Then $\mc(2,q)$ has no odd codewords on fewer than $\frac q p + 2$ lines.
\end{crl}

\begin{proof}
 Let $c$ be an odd codeword whose support support is contained in the union of $n$ lines $\ell_1, \dots, \ell_n$ through a point $P$.
 By \Cref{Res:MinWt}, $n \geq 2$.
 We can choose coordinates such that $\ell_1$ has equation $Y=0$ and $\ell_2$ has equation $X=0$.
 By \Cref{Thm:ConcurrentCodeword},
 \[
 F_c(X,Y,Z) = \prod_{\overline d \in \overline D}(Y-\overline d X) G(X,Y,Z) + \alpha (1-X^{q-1}), 
 \]
 with $\alpha \in \FF_q$, $\overline D$ a set of size $q+1-n$ and $G$ a homogeneous polynomial of degree $n-2$.
 Since $c$ is not constant on the points of $\ell_1 \setminus \{P\}$, the polynomial $f(Z) = F_c(1,0,Z)$ is not constant.
 Note that $\deg f(Z) = \deg G(1,0,Z) \leq \deg G(X,Y,Z) = n-2$.
 On the other hand, $f(Z)$ only takes values in $\FF_p$, hence $f(Z)^p - f(Z)$ is identically zero.
 Since $f(Z)$ is not constant, this means that $\deg(f(Z)^p - f(Z)) = p \deg f(Z) \geq q$.
 We conclude that $n \geq \deg f(Z) + 2 \geq \frac q p + 2$.
\end{proof}

Note that $q/p+2$ is at least $3$.
Hence, the above corollary tells us that codewords whose support is contained in $2$ lines are linear combinations of these lines, which is easy to verify combinatorially.
Note however that in case $q=p$ this bound is tight, as illustrated by the odd codewords on 3 lines.

\bigskip

Next, we introduce the concept of the rank of a codeword and use it to prove that there are odd codewords on $n$ lines, which can't be written in a trivial way as sums of odd codewords on fewer than $n$ lines.

\begin{df}
Let $c$ be a codeword of $\mc(2,q)$ and $P$ a point of $\pg(2,q)$.
We define the \emph{$P$-rank} $\rk_P(c)$ of $c$ as the integer $r$ such that after a projective transformation that maps $P$ to $(0,0,1)$, the polynomial $F_c$ related to $c$ (as defined in the beginning of this section) has $Z$-degree $r$.
\end{df}

\begin{rmk}
 \label{Rmk:Rank}
\begin{enumerate}
 \item The $P$-rank of $c$ is well defined, since a projective transformation which fixes $(0, 0, 1)$ does not change the $Z$-degree of $F_c(X, Y, Z)$.
 \item \label{Rmk:Rank:Concurrent} If $\supp(c)$ is covered by $n$ concurrent lines $\ell_1, \dots, \ell_n$ through $P$, then after a coordinate transformation that maps $P$ to $(0,0,1)$ and $\ell_1$ to the line $X=0$, $\rk_P(c)$ equals the $Z$-degree of the polynomial $G$ as defined in \Cref{Thm:ConcurrentCodeword}.
 In particular, this means that $\rk_P(c) \leq n-2$ if $n\geq 2$.
 \item \label{Rmk:Rank:Sum} It follows immediately from the definition of the rank that for any two codewords $c_1$ and $c_2$, $\rk_P(c_1+c_2) \leq \max\{\rk_P(c_1), \rk_P(c_2)\}$.
\end{enumerate}
\end{rmk}

\begin{prop}
 \label{Crl:ConcurrentCodeword}
 Let $p$ be prime and consider $n \geq 1$ concurrent lines $\ell_1, \dots, \ell_n$ in $\pg(2,p)$ through the point $P$.
 Let $r$ be an integer with $0 \leq r \leq n-2$.
 Then the set
 \[
  C(P,\{\ell_1,\dots,\ell_n\},r) = \sett{c \in \mc(2,p)}{\supp(c) \subseteq \bigcup_{i=1}^n \ell_i,\, \rk_P(c) \leq r}
 \]
 is a subspace of $\mc(2,p)$ of dimension $(r+1)\left(n-1-\frac r2\right)+1$.
 In particular, the subspace of all codewords whose support is covered by $\ell_1, \dots, \ell_n$ has dimension $\binom n2 + 1$.
\end{prop}

\begin{proof}
 Verifying that $C(P,\{\ell_1,\dots,\ell_n\},r)$ is a subspace is easy, since both the condition that $\supp(c)$ is a subset of a certain set and the condition $\rk_P(c) \leq r$ are preserved when taking linear combinations.
 
 We may suppose that the point of concurrency is $(0,0,1)$ and $\ell_n$ is the line $X=0$.
 Then there is a set $D = \{d_1, \dots, d_{n-1}\} \subseteq \FF_p$ such that $\ell_i$ has equation $Y = d_i X$ for $i < n$.
 Now we can apply \Cref{Thm:ConcurrentCodeword}.
 The condition that $F$ only takes values in $\FF_p$ holds trivially.
 Therefore, we are free to choose the coefficient $F(0,0,0)$ and a homogeneous polynomial $G(X,Y,Z)$ of degree $n-2$ which has $Z$-degree at most $r$.
 Such polynomials $G$ form a subspace of dimension
 \[
  \sum_{i=0}^r n-2-i = (r+1)\left( n-1 - \frac r2 \right).
 \]
 The last part of the theorem follows from the fact that each codeword whose support is covered by the lines $\ell_1, \dots, \ell_n$ has rank at most $n-2$ by \Cref{Rmk:Rank} (\ref{Rmk:Rank:Concurrent}).
\end{proof}

\begin{prop}
 \label{Prop:ReducibleOddCodeword}
 Let $c$ be a codeword of $\mc(2,p)$ whose support is covered by the $n$ concurrent lines through the common point $P$.
 Then $\rk_P(c) \leq r$ if and only if $c$ is a linear combination of codewords whose support is covered by $r+2$ lines through $P$.
\end{prop}

\begin{proof}
 First suppose that $c$ is a linear combination of codewords whose support is covered $r+2$ lines through $P$.
 Then each such codeword has $P$-rank at most $r$ by \Cref{Rmk:Rank} (\ref{Rmk:Rank:Concurrent}).
 Hence, $c$ has $P$-rank at most $r$ by \Cref{Rmk:Rank} (\ref{Rmk:Rank:Sum}).

 Conversely, suppose that $\rk_P(c) \leq r$.
 Let $\ell_1, \dots, \ell_n$ be the lines through $P$ that cover $\supp(c)$.
 If $r=n-2$, the statement of the proposition is trivial.
 If $r<n-2$, then using the notation of \Cref{Crl:ConcurrentCodeword}, define $C = C(P,\{\ell_1,\dots,\ell_n\},r)$ and for $j=1,2$ define $C_j = C(P,\{\ell_j,\ell_3,\dots,\ell_n\},r)$.
 We will prove that the subspace $\vspan{C_1,C_2}$ spanned by $C_1$ and $C_2$ equals $C$.
 As we already observed, $\vspan{C_1,C_2} \leq C$, hence it suffices to prove that both spaces have the same dimension.
 Note that $C_1 \cap C_2 = C(P,\{\ell_3,\dots,\ell_n\},r)$.
 Hence, by Grassmann's identity and \Cref{Crl:ConcurrentCodeword},
 \begin{align*}
  \dim(\vspan{C_1,C_2}) 
   &= \dim C_1 + \dim C_2 - \dim(C_1 \cap C_2) \\
   &= 2\left( (r+1)\left(n-2-\frac r2\right) + 1 \right) - \left( (r+1)\left(n-3-\frac r2\right) + 1 \right) \\
   &= (r+1)\left(n-1-\frac r2\right) + 1 = \dim C.
 \end{align*}
 By inductively applying this argument, we obtain that $C$ is spanned by the subspaces $C(P,\{\ell_j,\ell_{n-r},\dots,\ell_n\},r)$ for $j=1,\dots,n-r-1$.
 Since $c \in C$, this implies that $c$ can be written as a sum of codewords of $P$-rank at most $r$, whose support is covered by a subset of $\ell_1, \dots, \ell_n$ of size $r+2$.
\end{proof}

\begin{crl}
\label{odd codewords of maximal rank}
 Let $p$ be prime and choose an integer $n$ with $3 \leq n \leq p+1$.
 Choose $n$ concurrent lines $\ell_1, \dots, \ell_n$ with common point $P$.
 Then $\mc(2,p)$ contains odd codewords on the $n$ lines $\ell_1, \dots, \ell_n$ which cannot be written as linear combinations of codewords whose support is covered by fewer than $n$ lines through $P$.
\end{crl}

\begin{proof}
 Using \Cref{Thm:ConcurrentCodeword}, we can construct odd codewords $c$ on $\ell_1, \dots, \ell_n$ with $P$-rank $n-2$.
 Such a codeword $c$ satisfies the requirement of the corollary by \Cref{Prop:ReducibleOddCodeword}.
\end{proof}

\subsection{A consequence regarding multisets}

Using the link between codewords and multisets described in \Cref{Sec:SetsToCode}, one may translate \Cref{Prop:ReducibleOddCodeword} and \Cref{odd codewords of maximal rank} into the following corollary on multisets in $\ag(2,p)$.

\begin{crl}
 \label{Crl:MultisetNotUnion}
 Let $p$ be prime.
 Choose a subset $D$ of $\FF_p \cup \{\infty\}$ of size $n \geq 3$.
 Then there exists a multiset of points in $\ag(2,p)$ whose set of mod-special directions is $D$ and which is not a union of multisets with fewer than $n$ mod-special directions.
Furthermore, any multiset with $n$ mod-special directions, whose projection function in any direction has degree at most $r < n-2$, is necessarily the union of multisets with at most $r+2$ mod-special directions.    
\end{crl}

\section{Small weight codewords of \texorpdfstring{$\boldsymbol{\mathcal C(2,p)}$}{C(2,p)}}
 \label{Sec:Small}

The goal of this section is to prove \Cref{Thm:SmallWeight}, which is stated again below.

\bigskip

\noindent {\bf \Cref{Thm:SmallWeight}.} {\it
Suppose that $p \geq 37$ is prime and that $c \in \mc(2,p)$ with
\[
 \wt(c) \leq \begin{cases}
  4p+3 & \text{if } p \leq 47, \\
  5p-36 & \text{if } p \geq 53.
 \end{cases}
\]
 Then either
 \begin{enumerate}
  \item $c$ is a linear combination of at most 4 lines,
  \item $c$ is a linear combination of an odd codeword on 3 lines and another line,
  \item or $c$ is an odd codeword on 4 lines.
 \end{enumerate}
}

We will use the following results by the two last authors of this paper.

\begin{res}[{\cite[Theorem 4.2, Corollary 4.9]{Szonyi:Weiner:14}}]
\label{SzW}
Suppose that $q \geq 19$ and $c \in \mc(2,q)$.
 \begin{enumerate}
  \item If $\wt(c) < \sqrt{\frac q2} (q+1)$, then the points of $\supp(c)$ can be covered by $\ceil{\frac{\wt(c)}{q+1}}$ lines.
  \item For any integer $k$ with $0 \leq k < \sqrt{\frac q2} - 1$, either $\wt(c) \leq k q + 1$ or $\wt(c) \geq (k+1)q - \frac{3k^2 + 5k + 2}2$.
 \end{enumerate}
\end{res}

In addition, we will apply a simple, but very useful lemma.
Let $\mp$ be the set of points of $\pg(2,q)$ and let $c_1, c_2$ be two functions $\mp \to \FF_p$.
Define the dot product of $c_1$ and $c_2$ as
\[
 c_1 \cdot c_2 = \sum_{P \in \mp} c_1(P) c_2(P).
\]
Note in particular that for a line $\ell$, $c_1 \cdot \chi_\ell = \sum_{P \in \ell} c(P)$.

Denote the constant function taking the value 1 everywhere as $\one$.

\begin{lm}[{\cite[Lemma 2]{LavrauwStormeVandeVoorde:08:I}}]
 Suppose $c \in \mc(2,q)$.
 Then for every line $\ell$, $c \cdot \chi_\ell = c \cdot \one$.
\end{lm}

Now assume that $p$ and $c$ meet the conditions of \Cref{Thm:SmallWeight}.
By \Cref{SzW}, the points of $\supp(c)$ can be covered by 4 lines $\ell_1, \dots, \ell_4$ (it is possible that not all of these lines are necessary).
We distinguish 3 cases, and prove that in each of these cases, $c$ must be of one of the 3 types described in \Cref{Thm:SmallWeight}.

\paragraph{Case 1: $\ell_1, \dots, \ell_4$ are concurrent.}
Then it follows directly from \Cref{Thm:ConcurrentCodeword} and the definition of odd codewords that $c$ is either a linear combination of $\ell_1, \dots, \ell_4$, an odd codeword on these 4 lines, or a linear combination of an odd codeword on 3 lines and another concurrent line.

\paragraph{Case 2: Three of the lines are concurrent.}

Suppose that the point $P$ lies on 3 of the lines covering $\supp(c)$, w.l.o.g.\ $\ell_1,\ell_2,\ell_3$.
Let $P_i$ denote $\ell_i \cap \ell_4$.
Take a point $Q \neq P_i$ on $\ell_4$.
Then 
\[
 c(Q) = c \cdot \chi_{\vspan{P,Q}} - c(P) = c \cdot \one - c(P).
\]
In particular, $c(Q)$ is the same for all these points $Q$.
This means that $c - (c \cdot \one - c(P)) \chi_{\ell_4}$ is a codeword of $\mc(2,p)$ whose support is contained in 3 concurrent lines, and hence is either a linear combination of these 3 lines or an odd codeword on these 3 lines.

\paragraph{Case 3: The lines are in general position.}

We will prove that $c$ is a linear combination of the lines $\ell_1,\dots, \ell_4$.
We may suppose that $\ell_1$ has equation $X=0$, $\ell_2$ has equation $Y=0$, $\ell_3$ has equation $Z=0$, and $\ell_4$ has equation $X+Y+Z=0$.
Denote the points corresponding to the standard basis vectors $(1,0,0), (0,1,0), (0,0,1)$ as $E_1,E_2,E_3$.
Now consider the system of linear equations
\[
 \begin{cases}
  \alpha_1 + \alpha_2 = c(E_3) \\
  \alpha_2 + \alpha_3 = c(E_1) \\
  \alpha_3 + \alpha_1 = c(E_2) \\
  \alpha_1 + \alpha_2 + \alpha_3 + \alpha_4 = c \cdot \one
 \end{cases}
\]
in the variables $\alpha_i$.
In odd characteristic, this system has a unique solution.
Now consider $v = c - \sum_i \alpha_i \chi_{\ell_i} \in \mc(2,p)$.
Then $v(E_1) = v(E_2) = v(E_3) = v \cdot \one = 0$.
Consider the points $Q_i = \ell_i \cap \ell_4$ for $i = 1,2,3$.
Then also $v(Q_i) = v \cdot \chi_{\vspan{Q_i,E_i}} = 0$.

If $\pi \in \Sym(3)$ is a permutation, then $\pi$ acts on $\FF_p^3$ by sending the vector $x = (x_0,x_1,x_2)$ to $x^\pi = \left( x_{\pi(0)},x_{\pi(1)},x_{\pi(2)} \right)$.
Therefore, we can interpret $\pi$ as an element of $\PGL(2,p)$.

\begin{lm}
 \label{Lm:Permutation}
 If $0 \in \{x,y,z\}$, then $v(x,y,z) = \sgn(\pi) \, v((x,y,z)^\pi)$.
\end{lm}

\begin{proof}
 First suppose that $x=0$.
 If one of the cases $y=0$, $z=0$, or $y=-z$ holds, then both sides of the equation are zero, so the equation holds.
 Otherwise, take the line $\ell = \vspan{(0,y,z),Q_3}$.
 Then $\ell$ can only intersect $\supp(c)$ in $(0,y,z)$ and $(y,0,z)$.
 Therefore $v(0,y,z) = -v(y,0,z)$.
 
 Now return to the general case (where the element of $x,y,z$ equal to zero is not necessarily $x$).
 Using similar arguments as above, we see that $v(x,y,z) = \sgn(\pi) \, v((x,y,z)^\pi)$ holds whenever $\pi$ is a transposition where one of the transposed elements equals $0$.
 The general statement of the lemma follows since these transpositions generate $\Sym(3)$.
\end{proof}

\begin{lm}
 \label{Lm:Coord}
 There exists a function $f: \FF_p \to \FF_p$ such that the following properties hold for all $z \in \FF_p$:
 \begin{enumerate}
  \item \label{Lm:Coord:Perm} $v(0,1,z) = v(1,z,0) = v(z,0,1) = f(z)$.
  \item \label{Lm:Coord:Inv} If $z \neq 0$, then $f(z^{-1}) = - f(z)$.
  \item \label{Lm:Coord:Roots} $f(0) = f(1) = f(-1) = 0$.
  \item \label{Lm:Coord:L4} $v(1,-z-1,z) = f(z)$.
  \item \label{Lm:Coord:Plus} $f(z) = -f(-z-1)$.
 \end{enumerate}
\end{lm}

\begin{proof}
 For each $z \in \FF_p$, define $f(z) = v(0,1,z)$.
 
 (\ref{Lm:Coord:Perm}) This follows from the definition of $f$ and \Cref{Lm:Permutation}.

 (\ref{Lm:Coord:Inv}) Suppose that $z \in \FF_p^*$.
 Then $(0,1,z^{-1})$ and $(0,z,1)$ represent the same point of $\pg(2,p)$.
 By applying \Cref{Lm:Permutation}, we see that
 \[
  f(z^{-1}) = v(0,1,z^{-1}) = v(0,z,1) = -v(0,1,z) = -f(z).
 \]

 (\ref{Lm:Coord:Roots}) By definition $f(0) = v(E_2) = 0$.
 For $z = \pm 1$, $z = z^{-1}$.
 By (\ref{Lm:Coord:Inv}), $f(z) = -f(z)$, which implies that $f(z) = 0$.

 (\ref{Lm:Coord:L4}, \ref{Lm:Coord:Plus})
 Consider the point $P_z = (1,-z-1,z)$ on the line $\ell_4$.
 First suppose that $z \notin \{0,-1\}$, so that $P_z$ is not one of the points $Q_i$.
 Consider the line $\ell = \vspan{P_z,E_2}$.
 Then $0 = v \cdot \chi_\ell = v(P_z) + v(1,0,z)$, hence $v(P_z) = -v(1,0,z)$.
 By \Cref{Lm:Permutation}, $v(P_z) = v(z,0,1)$.
 By (\ref{Lm:Coord:Perm}), $v(P_z) = f(z)$.
 This proves (\ref{Lm:Coord:L4}).
 
 Now let $\ell$ be the line $\vspan{P_z,E_1}$.
 Then $0 = v \cdot \chi_\ell = v(P_z) + v(1,-z-1,0)$.
 By (\ref{Lm:Coord:Perm}), $v(P_z) = -f(-z-1)$.
 Therefore, $f(z) = v(P_z) = -f(-z-1)$.
 This proves (\ref{Lm:Coord:Plus}).

 Finally, suppose that $z \in \{0,-1\}$.
 Then $P_z$ is one the points $Q_i$, hence $v(P_z) = 0$ and $f(z) = 0$ by (\ref{Lm:Coord:Roots}), proving (\ref{Lm:Coord:L4}).
 In addition, (\ref{Lm:Coord:Plus}) follows directly from (\ref{Lm:Coord:Roots}).
\end{proof}

\begin{lm}
 \label{Lm:Zero}
 The function $f$ from \Cref{Lm:Coord} must be zero everywhere.
\end{lm}

\begin{proof}
 Take $a \in \FF_p \setminus \{0,\pm 1\}$.
 Consider the line $\ell$ with equation $X + (a+1)Y + aZ = 0$.
 Since $a \notin \{0,\pm 1\}$, $\ell$ contains none of the points $E_i$ or $Q_i$, and therefore intersects each of the lines $\ell_1, \dots, \ell_4$ in a different point.
 We find that
 \[
  v \cdot \chi_\ell = v(0,1,-1-a^{-1}) + v(-a,0,1) + v(1,-(a+1)^{-1},0) + v(1,a-1,-a) = 0.
 \]
 \Cref{Lm:Coord} tells us that
 \begin{align*}
  0 &= f(-1-a^{-1}) + f(-a) + f(-(a+1)^{-1}) + f(-a) \\
  &= -f(a^{-1}) + f(-a) - f(-a-1) + f(-a) \\
  &= f(a) + f(-a) + f(a) + f(-a)
  = 2 (f(a) + f(-a)).
 \end{align*}
 It follows that $f(-a) = -f(a)$.
 This also holds for $a \in \{0, \pm 1\}$ by \Cref{Lm:Coord} (\ref{Lm:Coord:Roots}).
 Combining this with \Cref{Lm:Coord} (\ref{Lm:Coord:Plus}), this means that $f(z) = f(z+1)$ for all $z \in \FF_p$.
 It follows that $f$ is a constant function.
 By \Cref{Lm:Coord} (\ref{Lm:Coord:Roots}), $f$ must be zero everywhere.
\end{proof}

It follows that $v$ is zero everywhere.
Indeed, $\supp(v)$ is contained in the union of the lines $\ell_1, \dots, \ell_4$.
If $P$ is an intersection point of two of these lines (i.e.\ a point $E_i$ or $Q_i$), then we already know that $v(P) = 0$.
If $P$ lies on exactly one of the lines, then by \Cref{Lm:Coord} (\ref{Lm:Coord:Perm}) or (\ref{Lm:Coord:L4}) it follows that $v(P) = f(z)$ for some $z \in \FF_p$, hence $v(P) = 0$ by \Cref{Lm:Zero}.
We conclude that $c = \sum_i \alpha_i \chi_{\ell_i}$, hence $c$ is a linear combination of at most 4 lines.

\section{Odd codewords as linear combinations of lines}
 \label{Sec:LinComb}

Let $c$ be an odd codeword on some concurrent lines.
Since $c$ is a function from the points of $\pg(2,q)$ to $\FF_p$, we can also interpret it as a vector over $\FF_p$ whose positions are labelled by the points of $\pg(2,q)$.
Let $A$ denote the incidence matrix of $\pg(2,q)$.
This means that the rows and columns of $A$ are labelled by the lines and points respectively of $\pg(2,q)$, and that
\[
 A(\ell,P) = \begin{cases}
  1 & \text{if } P \in \ell, \\
  0 & \text{if } P \notin \ell.
 \end{cases}
\]
Writing $c$ as an explicit linear combination of lines is equivalent to finding a vector $v$ over $\FF_p$, whose positions are labelled by the lines of $\pg(2,q)$, with
\[
 v A = c.
\]
The reason why it is difficult to find such a $v$ is that over $\FF_p$, $A$ does not have full rank, and hence is not invertible.
However, we if we move to the rationals, then we find the well-known identity
\[
 A A^\top = q I + J,
\]
with $I$ the identity matrix, and $J$ the all-one matrix.
Since also $AJ = (q+1)J$,
\[
 A^{-1} = \frac1q\left( A^\top - \frac1 {q+1} J \right).
\]
This gives us a way finding a vector $v$ with $vA = c$ as follows.
\begin{itemize}
 \item[] {\bf Step 1.} Translate the codeword $c$ to a vector $c_\QQ$ over $\QQ$.
 This means the following.
 Interpret $\FF_p$ as $\ZZ / p \ZZ$.
 Then each element of $\FF_p$ is a coset of $p \ZZ$.
 For each point $P$, let $c_\QQ(P)$ be an integer in the coset $c(P)$ of $p \ZZ$.
 \item[] {\bf Step 2.} Calculate 
 \[
  v_\QQ = c_\QQ A^{-1} = \frac1q\left( c_\QQ A^\top - \frac1 {q+1} c_\QQ J \right).
 \]
 \item[] {\bf Step 3.} In case that every entry of $v_\QQ$ is a rational number whose denominator is not divisible by $p$, every entry corresponds to an element of $\FF_p$.
 We can translate $v_\QQ$ to a vector $v$ over $\FF_p$.
 This vector must satisfy $v A = c$.
 \item[] {\bf Step 4.} It also holds that $(v+v')A = c$ for each vector $v'$ in the left kernel of $A$.
 We can for instance take $v'$ to be the vector taking value 1 on the affine points, and value 0 on the points of $\ell_\infty$.
\end{itemize}

In the following subsections, using the approach described above, we explicitly construct odd codewords on $3$ and $4$ lines with maximum possible rank.

\subsection{An odd codeword on 3 lines}

 Let $p > 2$ be prime.
 Consider the odd codeword $c$ defined by
 \[
 c(P) = \begin{cases}
  z & \text{if } P =(1,0,-z) \text{ or } P = (0,1,z) \text{ or } P=(1,-1,z), \\
  0 & \text{otherwise,}
 \end{cases}
\]
on the lines $X=0$, $Y=0$, and $X+Y=0$.
Recall the function $\nu$ defined in \Cref{Not:Nu}.
Following Step 1, we can define $c_\QQ$ by
\[
 c_\QQ(P) = \begin{cases}
  - \nu(z) & \text{if } P = (1,0,z), \\
  \nu(z) & \text{if } P = (0,1,z) \text{ or } P = (1,-1,z), \\
  0 & \text{otherwise}.
 \end{cases}
\]
We calculate $c_\QQ A^\top$.
Take a line $\ell$ in $\pg(2,p)$.
First suppose that $\ell$ has an equation of the form $Z = aX + bY$.
Then $\ell$ intersects $\supp(c)$ in the points
\begin{align*}
 (1,0,a), && (0,1,b), && (1,-1,a-b).
\end{align*}
Therefore,
\[
 (c_\QQ A^\top)(\ell) = -\nu(a) + \nu(b) + \nu(a-b)
\]
Note that
\[
 \nu(a-b) = \begin{cases}
  \nu(a) - \nu(b) & \text{if } \nu(a) \geq \nu(b), \\
  \nu(a) - \nu(b) + p & \text{if } \nu(a) < \nu(b).
 \end{cases}
\]
Hence,
\[
 (c_\QQ A^\top)(\ell) = \begin{cases}
  p & \text{if } \nu(a) < \nu(b), \\
  0 & \text{if } \nu(a) \geq \nu(b).
 \end{cases}
\]
Now suppose that $\ell$ has an equation of the form $aX + bY = 0$.
Note that
\[
 \sum_{z \in \FF_p} \nu(z) = \sum_{i=0}^{p-1} i = \binom p 2.
\]
Thus, it is easy to check that
\[
 (c_\QQ A^\top)(\ell) = \begin{cases}
  - \binom p2 & \text{if } \vspan{(a,b)} = \vspan{(0,1)}, \\
  \binom p2 & \text{if } \vspan{(a,b)} = \vspan{(1,0)} \text{or } \vspan{(a,b)} = \vspan{(1,1)}, \\
  0 & \text{otherwise.}
 \end{cases}
\]
Moreover,
\[
 c_\QQ J = (2 - 1) \sum_{z \in \FF_p} \nu(z) = \binom p 2.
\]
If we define $v_\QQ = c_\QQ A^{-1}$, then
\[
 v_\QQ(\ell) = \begin{cases}
  1 - \frac{p-1}{2(p+1)} & \text{if $\ell = [a,b,-1]$ with } \nu(a) < \nu(b), \\
  - \frac{p-1}2 - \frac{p-1}{2(p+1)} & \text{if } \ell = [0,1,0], \\
  \frac{p-1}2 - \frac{p-1}{2(p+1)} & \text{if } \ell = [1,0,0] \text{ or } \ell = [1,1,0], \\
  - \frac{p-1}{2(p+1)} & \text{otherwise}. \\
 \end{cases}
\]
Now we translate this to a vector $v$ over $\FF_p$.
As noted in Step 4, we are allowed to subtract a constant from the entries corresponding to the affine points.
This gives us the vector $v$ with $v A = c$ defined by
\[
 v(\ell) = \begin{cases}
  1 & \text{if $\ell = [a,b,-1]$ with } \nu(a) < \nu(b), \\
  0 & \text{if $\ell = [a,b,-1]$ with } \nu(a) \geq \nu(b), \\
  1 & \text{if } \ell = [0,1,0], \\
  0 & \text{if } \ell = [1,0,0] \text{ or } \ell = [1,1,0], \\
  \frac12 & \text{otherwise}. \\
 \end{cases} 
\]

\begin{rmk}
 Essentially the same vector $v$ was also found in an unpublished manuscript of De Boeck and Vandendriessche \cite{DeBoeckVandendriessche:Unp}.
 Their proof relies on a much more technical calculation.
\end{rmk}

We can now give a different proof for the fact that the set of points in \Cref{Res:KissSomlai} has 3 special directions.
We will follow the notation of \Cref{Res:KissSomlai} and simply write $a < b$ instead of $\nu(a) < \nu(b)$.

\begin{thm}
 \label{Thm:KissSomlaiAgain}
 Let $p > 2$ be prime. Then the set
 \[
  S = \sett{(x,y) \in \FF_p^2}{ y < x }
 \]
 has exactly 3 special directions.
\end{thm}

\begin{proof}
 Consider the odd codeword $c$ defined at the beginning of this subsection.
 We have found the vector $v$ which writes $c$ as a linear combination of lines.
 Consider the set
 \[
  S' = \sett{(a,b) \in \FF_q^2}{-a < -b}.
 \]
 By \Cref{Prop:CodeToSet}, $S'$ has exactly 3 mod-special directions, and all lines with a mod-equidistributed direction meet $S'$ in $\nu(-1/2) = \frac{p-1}2$ points modulo $p$.
 Since $S'$ is an actual set, each lines meets $S'$ in at least 0 and at most $p$ points, so all lines with a mod-equidistributed direction meet $S'$ in exactly $\frac{p-1}2$ points.
 Therefore, all mod-equidistributed directions are actually equidistributed, and $S'$ has exactly 3 special directions.
 Since $S'$ is the image of the set $S$ from the theorem after the affine transformation $(X,Y) \mapsto (-Y,-X)$, $S$ also has exactly 3 special directions.
\end{proof}

\subsection{An odd codeword on 4 lines}
 \label{Sec:4Lines}

Let $p > 3$ again be prime.
Consider the odd codeword $c$ from \Cref{Ex:OddOn4} given 
 \[
  c(P) = \begin{cases}
   2 z^2 & \text{if $P = (1,0,z)$ or $P=(0,1,z)$}, \\
   - z^2 & \text{if $P = (1,1,z)$ or $P = (1,-1,z)$}, \\
   0 & \text{otherwise}.
  \end{cases}
 \]
on the 4 lines $X=0$, $Y=0$, $X=Y$, and $X=-Y$.
Following Step 1, we can choose $c_\QQ$ to be
\[
 c_\QQ(P) = \begin{cases}
  2 \nu(z)^2 & \text{if } P = (1,0,z) \text{ or } P = (0,1,z), \\
  - \nu(z)^2 & \text{if } P = (1,1,z) \text{ or } P = (1,-1,z), \\
  0 & \text{otherwise.}
 \end{cases}
\]
Now we calculate $c_\QQ A^\top$.
First take a line $\ell$ with equation $Z = aX + bY$.
Then $\ell$ can only intersect $\supp(c)$ in the points
\begin{align*}
 (1,0,a), && (0,1,b), && (1,1,a+b), && (1,-1,a-b).
\end{align*}
Therefore,
\begin{align*}
 (c_\QQ A^\top)(\ell)
 =& 2 \nu(a)^2 + 2 \nu(b)^2 - \nu(a+b)^2 - \nu(a-b)^2 \\
 =& 2 \nu(a)^2 + 2\nu(b)^2 
  - \begin{cases}
   (\nu(a) + \nu(b))^2 & \text{if } \nu(a) + \nu(b) < p \\
   (\nu(a) + \nu(b) - p)^2 & \text{if } \nu(a) + \nu(b) \geq p
  \end{cases} \\
  &- \begin{cases}
   (\nu(a) - \nu(b))^2 & \text{if } \nu(a) \geq \nu(b) \\
   (\nu(a) - \nu(b) + p)^2 & \text{if } \nu(a) < \nu(b)
  \end{cases} \\
  =& \begin{cases}
   0 & \text{if $\nu(a) + \nu(b) < p$ and $\nu(a) \geq \nu(b)$}, \\
   p( 2\nu(a) + 2\nu(b) - p) & \text{if $\nu(a) + \nu(b) \geq p$ and $\nu(a) \geq \nu(b)$}, \\
   p(2\nu(a) - 2\nu(b) - p) & \text{if $\nu(a) + \nu(b) < p$ and $\nu(a) < \nu(b)$}, \\
   p(4\nu(a)-2p) & \text{if $\nu(a) + \nu(b) \geq p$ and $\nu(a) < \nu(b)$}, \\
  \end{cases}
\end{align*}

We remark that
\[
 \sum_{z \in \FF_p} \nu(z)^2 = \sum_{i=1}^{p-1} i^2 = \frac{(p-1)p(2p-1)}{6}.
\]
It is then straightforward to finish the calculation of the vector $v$, similar to the previous example.
We find
\[
 v(\ell) = \begin{cases}
  0 & \text{if $\ell = [a,b,-1]$ with $\nu(a) + \nu(b) < p$ and $\nu(a) \geq \nu(b)$}, \\
  2(a+b) & \text{if $\ell = [a,b,-1]$ with $\nu(a) + \nu(b) \geq p$ and $\nu(a) \geq \nu(b)$}, \\
  2(a-b) & \text{if $\ell = [a,b,-1]$ with $\nu(a) + \nu(b) < p$ and $\nu(a) < \nu(b)$}, \\
  4a & \text{if $\ell = [a,b,-1]$ with $\nu(a) + \nu(b) \geq p$ and $\nu(a) < \nu(b)$}, \\
  0 & \text{if } \ell = [1,0,0] \text{ or } \ell = [0,1,0], \\
  - \frac 12 & \text{if } \ell = [1,1,0] \text{ or } \ell = [1,-1,0], \\
  - \frac 13 & \text{otherwise}.
 \end{cases}
\]
Note that the condition $p > 3$ is necessary for the above vector to make sense as taking values in $\FF_p$, since we have fractions with $2$ and $3$ in the denominator.

The vector $v$ yields a multiset of points in $\ag(2,p)$ with exactly 4 mod-special directions.
By doing some manipulations which do not alter the mod-special directions, we find the following multiset.

\begin{prop}
 Let $p > 3$ be prime.
 Consider the multiset $M$ of points in $\ag(2,p)$ where a point $(x,y)$ has multiplicity
 \[
  \begin{cases}
  \nu(-x) & \text{if $\nu(x) + \nu(y) < p$ and $x \geq y$,} \\
  \nu(y) & \text{if $\nu(x) + \nu(y) \geq p$ and $x \geq y$,} \\
  \nu(-y) & \text{if $\nu(x) + \nu(y) < p$ and $x < y$,} \\
  \nu(x) & \text{if $\nu(x) + \nu(y) \geq p$ and $x < y$.}
  \end{cases}
 \]
 Then $M$ has exactly $4$ mod-special directions, namely $(0)$, $(\infty)$, $(1)$, and $(-1)$.
 Moreover, $M$ is not the union of multisets with at most 3 mod-special directions.
\end{prop}

\section{Conclusion and open problems}
 \label{Sec:Conclusion}

In this paper, we drew a connection between multisets of points in $\ag(2,q)$ with $k$ mod-special directions and odd codewords of $\mc(2,q)$ on $k$ concurrent lines.
The astute reader will have noticed that the proofs in \Cref{Sec:Concurrent} are mostly the same as the proofs in the beginning of \Cref{Sec:ProofsOfTheorems}.
Using the link between multisets and codewords outlined in \Cref{Sec:SetsToCode}, one can translate proofs of one of these sections into to other.
As a consequence of the link between the two topics, \Cref{Thm:ConcurrentCodeword} can be translated into a necessary and sufficient condition for a polynomial $F(X,Y,Z)$ to be of the form $F_M$ for some multiset $M$, where $F_M$ denotes polynomial associated to $M$ as defined in \Cref{Prop:AddRedPol}.
This may serve as an extra motivation why the study of special and mod-special directions would benefit from its link with the study of odd codewords.

\bigskip

We end this paper by discussing some interesting open problems.
\begin{enumerate}
 \item \Cref{Sec:LinComb} outlines a procedure to construct multisets of points in $\ag(2,p)$ with few mod-special directions.
 We saw in \Cref{Sec:4Lines} how to construct a multiset with 4 mod-special directions.
 This set divides the plane into four parts, and in each of these parts assigns multiplicities to the points according to a linear function.
 If we would extrapolate this procedure starting from an odd codeword on $k$ lines, which is not a linear combination of odd codewords on fewer than $k$ lines with the same point of concurrency, we would expect to find multisets of points where the multiplicities are functions of degree $k-2$.
 It is therefore still a very hard problem to construct sets of points with exactly $k$ special directions.
 One might even wonder whether this is even possible.
 Does it hold that for each $k \geq 4$ and all sufficiently large primes $p$ (where the lower bound on $p$ depends on $k$), $\ag(2,p)$ has no sets of points with exactly $k$ special directions?
 \item In \Cref{Thm:MinSpecialDir}, we gave a lower bound on the minimum number of special directions of a set of size $nq$ in $\ag(2,q)$, with $q = p^h$, $h > 1$.
 We saw that this lower bound is tight if $h = 2$ and $n=1$.
 Are there other examples where this bound is tight?
 Can it be improved?
 \item We saw in \Cref{Res:SzWNonPrime} that if $q$ is large enough and not prime, then all codewords of $\mc(2,q)$ up to weight approximately $\sqrt q q$, or $\frac{\sqrt q}2 q$ if $q$ is the square of a prime, are linear combinations of few lines.
 In case $q$ is prime this no longer holds, as illustrated by the odd codewords on concurrent lines.
 Thus, we can make small weight codewords of $\mc(2,p)$ by choosing a small number $n$ of lines, and making a linear combination of the lines and odd codewords on concurrent subsets of these lines.
 Up to which weight are all codewords of this form?
 We saw that this is the case up to weight at least $5p-36$ for $p \geq 53$.
 However, our proof is too ad hoc to extend to higher weights.
 When classifying codewords whose support is contained in the union of 5 lines $\ell_1, \dots, \ell_5$, more complicated configurations arise.
 For example, if $\ell_1, \ell_2, \ell_3$ are concurrent and $\ell_3,\ell_4,\ell_5$ are concurrent at a different point, we can make a linear combination of and odd codeword on $\ell_1,\ell_2,\ell_3$ and an odd codeword on $\ell_3,\ell_4,\ell_5$.
 A more refined strategy will be necessary to deal with these kinds of configurations.
\end{enumerate}

\paragraph{Acknowledgements.} The first author is grateful to the others authors for their hospitality during his visit to Budapest.
The first author was partially supported by Fonds Wetenschappelijk Onderzoek project 12A3Y25N and by a Fellowship of the Belgian American Educational Foundation. The second author was partially supported by the Slovenian Research Agency,
research project J1-9110.

\newcommand{\etalchar}[1]{$^{#1}$}


\begin{thebibliography}{LSVdV10}

\bibitem[AD24]{Adriaensen:Denaux:24}
S.~Adriaensen and L.~Denaux.
\newblock Small weight codewords of projective geometric codes {II}.
\newblock {\em Des. Codes Cryptogr.}, 92(9):2451--2472, 2024.

\bibitem[AW24]{Parabola}
S.~Adriaensen and Zs. Weiner.
\newblock Points below a parabola in affine planes of prime order, 2024.
\newblock arXiv:2411.19202.

\bibitem[Bag12]{Bagchi}
B.~Bagchi.
\newblock On characterizing designs by their codes.
\newblock In {\em Buildings, Finite Geometries and Groups}, volume~10 of {\em Springer Proceedings in Mathematics}, pages 1--14. Springer, New York, 2012.

\bibitem[Bal03]{Ball03}
S.~Ball.
\newblock The number of directions determined by a function over a finite field.
\newblock {\em J. Combin. Theory Ser. A}, 104(2):341--350, 2003.

\bibitem[BBB{\etalchar{+}}99]{BBBSS}
A.~Blokhuis, S.~Ball, A.~E. Brouwer, L.~Storme, and T.~Sz{\H o}nyi.
\newblock On the number of slopes of the graph of a function defined on a finite field.
\newblock {\em J. Combin. Theory Ser. A}, 86(1):187--196, 1999.

\bibitem[BBW91]{blokhuisbrouwerwilbrink}
A.~Blokhuis, A.~Brouwer, and H.~Wilbrink.
\newblock Hermitian unitals are code words.
\newblock {\em Discrete Math.}, 97(1-3):63--68, 1991.

\bibitem[DB14]{DeBoeckPhd}
Maarten De~Boeck.
\newblock {\em Intersection Problems in Finite Geometries}.
\newblock Ph{D} thesis, Universiteit Gent, 2014.

\bibitem[DBV]{DeBoeckVandendriessche:Unp}
M.~De~Boeck and P.~Vandendriessche.
\newblock The weird code word.
\newblock unplublished manuscript.

\bibitem[DGM70]{delsartegoethalsmacwilliams}
Ph. Delsarte, J.-M. Goethals, and F.~J. MacWilliams.
\newblock On generalized {R}eed-{M}uller codes and their relatives.
\newblock {\em Information and Control}, 16:403--442, 1970.

\bibitem[FSSW08]{FSSzW}
S.~Ferret, L.~Storme, P.~Sziklai, and {Zs.} Weiner.
\newblock A {$t\pmod{p}$} result on weighted multiple {$(n-k)$}-blocking sets in {${\rm PG}(n,q)$}.
\newblock {\em Innov. Incidence Geom.}, 6/7:169--188, 2007/08.

\bibitem[Ghi20]{Ghidelli}
L.~Ghidelli.
\newblock On rich and poor directions determined by a subset of a finite plane.
\newblock {\em Discrete Mathematics}, 343(5):111811, 2020.

\bibitem[Ham68]{hamada1968}
N.~Hamada.
\newblock The rank of the incidence matrix of points and $ d $-flats in finite geometries.
\newblock {\em Journal of Science of the Hiroshima University, Series AI (Mathematics)}, 32(2):381--396, 1968.

\bibitem[KS24]{Kiss:Somlai}
G.~Kiss and G.~Somlai.
\newblock Special directions on the finite affine plane.
\newblock {\em Designs, Codes, and Cryptography}, 2024.

\bibitem[LSVdV08]{LavrauwStormeVandeVoorde:08:I}
M.~Lavrauw, L.~Storme, and G.~Van~de Voorde.
\newblock On the code generated by the incidence matrix of points and hyperplanes in {PG$(n,q)$} and its dual.
\newblock {\em Des. Codes Cryptogr.}, 48(3):231--245, 2008.

\bibitem[LSVdV10]{Lavrauw:Storme:VandeVoorde:10}
M.~Lavrauw, L.~Storme, and G.~Van~de Voorde.
\newblock Linear codes from projective spaces.
\newblock In {\em Error-correcting codes, finite geometries and cryptography}, volume 523 of {\em Contemp. Math.}, pages 185--202. Amer. Math. Soc., Providence, RI, 2010.

\bibitem[R{\'e}d73]{lacunary}
L.~R{\'e}dei.
\newblock {\em Lacunary polynomials over finite fields}.
\newblock North-Holland Publishing Co., Amsterdam-London; American Elsevier Publishing Co., Inc., New York, 1973.
\newblock Translated from the German by I. F\"oldes.

\bibitem[SW18]{Szonyi:Weiner:14}
T.~Sz\H{o}nyi and {Zs}. Weiner.
\newblock Stability of {$k\bmod p$} multisets and small weight codewords of the code generated by the lines of {PG}(2, {$q$}).
\newblock {\em J. Combin. Theory Ser. A}, 157:321--333, 2018.

\bibitem[Sz{\H o}96]{Szonyi96}
T.~Sz{\H o}nyi.
\newblock On the number of directions determined by a set of points in an affine {G}alois plane.
\newblock {\em J. Combin. Theory Ser. A}, 74(1):141--146, 1996.

\bibitem[Sz{\H o}99]{Szonyi99}
T.~Sz{\H o}nyi.
\newblock Around {R}\'edei's theorem.
\newblock {\em Discrete Math.}, 208/209:557--575, 1999.
\newblock Combinatorics (Assisi, 1996).

\end{thebibliography}
\end{document}